\documentclass[11pt]{amsart}
\usepackage[latin1]{inputenc}
\usepackage[english]{babel}
 \usepackage{cite}
 \usepackage{relsize}
 \usepackage{bm}
 \usepackage{bibentry}
 \usepackage{graphicx}
 \usepackage[all]{xypic}
 \usepackage{etoolbox}%
\usepackage[usenames,dvipsnames]{color}

\definecolor{orange}{rgb}{1,0.5,0}
\definecolor{myblue}{rgb}{0.75,0.85,1.00}
\definecolor{light-gray}{gray}{0.93}
\usepackage{amsmath, amsfonts, amssymb, amsthm, mathtools, mathrsfs,bm}
\usepackage{enumerate}
\usepackage[numbers]{natbib}
\usepackage{setspace}

\newcommand{\nsh}{{n\,\sharp}}
\def\Hom{\mathop{\rm Hom}\nolimits}

\usepackage[pagebackref,colorlinks,linkcolor=red,citecolor=blue,urlcolor=blue,hypertexnames=true]{hyperref}
\usepackage[pagebackref,hypertexnames=true]{hyperref}

\newtheorem{theorem}{Theorem}[section]
\newtheorem{lemma}[theorem]{Lemma}
\newtheorem{proposition}[theorem]{Proposition}
\newtheorem{corollary}[theorem]{Corollary}

\theoremstyle{definition}

\newtheorem{example}[theorem]{Example}
\newtheorem{remark}[theorem]{Remark}

\newcommand{\G}{\Gamma}

\newcommand{\EG}{\underline{E}\Gamma}

\newcommand{\QQ}{\mathcal{Q}}

\newcommand{\ep}{\varepsilon}

\newcommand{\Z}{\mathbb{Z}}
\newcommand{\Q}{\mathbb{Q}}
\newcommand{\R}{\mathbb{R}}
\newcommand{\C}{\mathbb{C}}

\begin{document}

\title{  Quasi-representations of groups and two-homology }
\author{Marius Dadarlat}\address{MD: Department of Mathematics, Purdue University, West Lafayette, IN 47907, USA}\email{mdd@purdue.edu}	
\begin{abstract}
The Exel-Loring formula asserts that two topological invariants associated to a pair of almost commuting unitary matrices coincide. Such a pair can be viewed as a quasi-representation of $\Z^2$. We give a generalization of this formula for countable discrete groups. We also show the nontriviality of the corresponding invariants for quasidiagonal groups which are coarsely embeddable in a Hilbert space and have nonvanishing second Betti number.
\end{abstract}
\thanks{This research was partially supported by NSF grant \#DMS--1700086}
\date{\today}
\maketitle
\section{Introduction}

Kazhdan \cite{Kazhdan-epsilon} and Voiculescu \cite{Voi:unitaries} exhibited sequences of pairs of almost commuting unitaries without commuting approximants. In their proofs, Kazhdan used a winding number argument and Voiculescu used a Fredholm index argument. Another proof was given later by Loring using K-theory \cite{Loring:k(uv)-paper}.
For two unitaries $u,v\in U(n)$ such that $\|uv-vu\|$ is smaller than a positive universal constant, Loring introduced  a K-theory invariant
$k(u,v)\in \Z$ which can be described informally as follows. The pair $u,v$
gives rise to a group quasi-representation $\varphi: \Z^2 \to U(n)$ and hence to a contractive quasi-representation of $*$-algebras $\varphi: \ell^1(\Z^2)\to M_n(\C)$. Then  $k(u,v)$ is defined as the pushforward $\varphi_\sharp(\beta)$ of the Bott element, where $K_0(\ell^1(\Z^2))\cong K_0(C^*(\Z^2)) \cong \Z\oplus\Z\beta$.  The virtual rank of $\beta$ is $0$ and the first Chern class of $\beta$ is $1$.   On the other hand, Exel and Loring \cite{Exel-Loring:winding}
rediscovered Kazhdan's invariant $\omega(u,v)$ defined as the winding number
in $\mathbb{C}\setminus\{0\}$ (abbreviated $\mathrm{wn}$) of the loop $t\mapsto det((1-t)1_n+t[v,u])$ and proved the equality $k(u,v)=\omega(u,v)$, \cite{Exel-Loring:inv=}.
 Exel gave another proof of this equality using the soft torus C*-algebra, see \cite{Exel:Pacific}.
   We extended the Exel-Loring formula  to quasi-representations $\pi:\Gamma_g \to U(n)$ of surface groups of genus $g\geq 1$ in \cite{BB} and in joint work with Carri\'on \cite{Carrion-Dadarlat}
  to  quasi-representations  $\rho: \Gamma_g \to U(A)$ for $A$ a unital tracial $C^*$-algebra, see Theorems~\ref{thm:bb}, \ref{thm:cc} below. A key step in these generalizations was to realize that the Exel-Loring formula is related to an index theorem of Connes, Gromov and Moscovici   \cite{CGM:flat} and to its extension studied in \cite{BB}.


By  Hopf's formula  $$H_2(\G,\Z)={R\cap [F,F]}/{[R,F]},$$
 for the second homology of a discrete group $\G$  in terms of a free presentation
\begin{equation}\label{eqn:res}
 0 \to R \to F \stackrel{q}{\longrightarrow} \G \to 0,\quad q(a)=\bar{a},
\end{equation}
 each element $x \in H_2(\G,\Z)$ is represented by a product of commutators $\prod_{i=1}^{g} [a_i,b_i]$
 with $a_i,b_i \in F$, for some integer $g\geq 1$, such that $\prod_{i=1}^{g} [\bar{a}_i,\bar{b}_i]=1$.

Consider the (rationally injective) homomorphism
$$\beta^{\G}:  H_2(\G,\Z) \cong H_2(B\G,\Z) \to RK_0(B\G),$$
 studied in \cite{Bettaieb-Matthey-Valette}, \cite{MR1951251},  \cite{MR2041902} and define
the map $\alpha^\G : H_2(\G,\Z) \to K_0(\ell^1(\G))$ as the composition $\alpha^\G=\mu_1^\G  \circ \beta^\G$ where $\mu_1^\G$ is  the $\ell^1$-version of the assembly map of \cite{Lafforgue}:
   \[\xymatrix{
\alpha^\G :H_2(\G,\Z)\ar[r]^-{\beta^\G}& RK_0(B\G) \ar[r]^{\mu_1^\G}&K_0(\ell^1(\G)).
}
\]
We generalize the Exel-Loring formula to arbitrary discrete countable groups $\G$ as follows.
\begin{theorem}\label{thm:index00}
 Let $\G$ be a discrete countable group. Let $x\in H_2(\G,\Z)$ be represented by  a product of commutators
 $\prod_{i=1}^{g} [a_i,b_i]$ with $a_i,b_i \in F$ and $\prod_{i=1}^{g} [\bar{a}_i,\bar{b}_i]=1$.  Let $p_0$ and $p_1$ be projections
  in some matrix algebra over $\ell^1(\Gamma)$ such that
  $\alpha^\G(x) = [p_0] - [p_1]\in K_0(\ell^1(\Gamma))$.
 There exist a finite set $S \subset G$  and $\varepsilon>0$ such that if $\pi:\G \to U(n)$ is unital map  with
 $\|\pi(st)-\pi(s)\pi(t)\|<\varepsilon$ for all $s,t \in S$, then
  \begin{equation}\label{eq:index}
  \pi_\sharp( \alpha^\G(x))=\mathrm{wn}
  \det \left((1-t)1_n+t\prod_{i=1}^{g} [\pi(\bar{a}_i),\pi(\bar{b}_i)] \right) =\frac{1}{2\pi i}\mathrm{Tr}\left( \log \left(\prod_{i=1}^{g} [\pi(\bar{a}_i),\pi(\bar{b}_i)]\right)\right).
 \end{equation}
 More generally if $A$ is a unital $C^*$-algebra with a trace $\tau$ and $\pi:\G \to U(A)$ is unital map  with
 $\|\pi(st)-\pi(s)\pi(t)\|<\varepsilon$ for all $s,t \in S$, then
  \begin{equation}\label{eq:indexc}
  \tau_*(\pi_\sharp( \alpha^\G(x)))=\frac{1}{2\pi i}\tau\left( \log \left(\prod_{i=1}^{g} [\pi(\bar{a}_i),\pi(\bar{b}_i)]\right)\right).
\end{equation}
Here $\pi_\sharp( \alpha^\G(x)) = \pi_\sharp(p_0) -
  \pi_\sharp(p_1)$ where $\pi_\sharp(p_i)$ is the K-theory class of the perturbation of $(\mathrm{id}\otimes \pi )(p_i)$ to a projection via analytic functional calculus.
\end{theorem}
 
 Moreover, we show in Theorem~\ref{thm:2}  that if $\G$  is a quasidiagonal group which admits a $\gamma$-element and $x\in H_2(\G,\Z)$ is not of finite order, then there are finite dimensional unitary quasi-representations $\pi:\G \to U(n)$ for which the winding number  of the closed loop
$$ t\mapsto \det \left((1-t)1_n+t\prod_{i=1}^{g} [\pi(\bar{a}_i),\pi(\bar{b}_i)] \right)$$
 from Theorem~\ref{thm:index00} is nonzero. In particular these quasi-representations are not perturbable to genuine representations,  see Corollary~\ref{cor:2}.
The proof of Theorem~\ref{thm:index00} combines results from \cite{BB},\cite{Carrion-Dadarlat} with results of Loday \cite{Loday} and Matthey \cite{MR1951251},  \cite{MR2041902}. For the proof of Theorem~\ref{thm:2} we rely on our previous paper \cite{CCC}.

 Eilers, Shulman  and  S{\o}rensen \cite{ESS-published} showed  that certain concrete groups with homogeneous relations are not matricially stable by using winding number invariants of Kazhdan/Exel-Loring type and  quasi-representations constructed ad-hoc.
Theorem~\ref{thm:index00} explains  how these invariants are connected to the two-homology of the groups and Theorem~\ref{thm:2} gives general abstract criteria for their nonvanishing.

\section{two-homology and winding numbers}
If $s,t$ are elements of a group $\G$, we denote by $[s,t]$ their commutator $sts^{-1}t^{-1}$. The commutator subgroup of $\G$, denoted $[\G,\G]$, consists of finite products of commutators.

If $\omega:[0,1] \to \C\setminus \{0\}$ is a loop, $\omega(0)=\omega(1)$, we denote by $\mathrm{wn}\,\omega(t)$ its winding number. Let $\log$ be the principal branch of the logarithm defined on $\C\setminus \{z\in \mathbb{R}\colon \, z \leq 0\}$, $\log 1 =0$.
Let $\mathrm{Tr}:M_n(\C) \to \C$ be the canonical trace with $\mathrm{Tr}(1_n)=n$.
Let $w\in SU(n)$ with $\|w-1_n\|<2$.   If $w$ is written as $w=\exp(2\pi i h)$ with $h=h^*=\frac{1}{2\pi i}\log (w)$, then $\mathrm{Tr}(h)\in \Z$ since $\det(w)=\exp(2\pi i \mathrm{Tr}(h))=1$.
Define the map $\kappa:\{w\in SU(n)\colon \|w-1\|<2\} \to \Z,$
\begin{equation}\label{eq:kappa} \kappa(w)=\frac{1}{2\pi i}\mathrm{Tr}(\log (w)).\end{equation}
The function $\kappa$ is continuous and hence locally constant as it assumes only integral values.

If $A$ is a unital
$C^*$-algebra with a trace $\tau$,
we define \mbox{$\kappa_\tau:\{w\in U(A)\colon \|w-1\|<1\} \to \R,$}
by
\begin{equation}\label{eq:kappatau} \kappa_\tau(w)=\frac{1}{2\pi i}\tau(\log (w)).\end{equation}
It is clear that if $A=M_n(\C)$ and $\tau=\mathrm{Tr}$, then $\kappa_\tau=\kappa$.
\begin{lemma}[\cite{Exel:Pacific}]
  If $w\in SU(n)$ and $\|w-1\|<2$, then $\mathrm{wn}\det\left((1-t)1_n+tw\right)=\kappa(w).$
\end{lemma}
\begin{proof}
  This is proved in \cite[Lemma 3.1]{Exel:Pacific} for a commutator $w=[u,v]$ with $u,v\in U(n)$. Let us review the argument.
  One verifies that if $h=h^*=\frac{1}{2\pi i}\log (w)$, then for all $0\leq t \leq 1$,
  $$\|\left((1-t)w^*+t1_n\right)-\exp(2\pi i th)w^*\|=\|\left((1-t)1_n+t \exp(2\pi i h) -\exp(2\pi i th)\right) \|<1.$$
  Thus the two paths $\omega_0(t)=(1-t)w^*+t1_n$ and $\omega_1(t)=\exp(2\pi i th)w^*$ are homotopic with endpoints fixed as maps into $GL(n,\C)$ via the linear homotopy $\omega_s(t)=(1-s)\omega_0(t)+s\omega_1(t)$.
  It follows that
  \[\mathrm{wn}\det\left((1-t)1_n+tw\right)=\mathrm{wn}\det(\exp(2\pi i th))=\mathrm{wn}\exp(2\pi i t\mathrm{Tr}(h)))= \mathrm{Tr}(h).\qedhere \]
\end{proof}

 \begin{lemma}[Lemma 5, \cite{Kazhdan-epsilon}]\label{lemma:Kaz}
  Let $(u_i)_{i=1}^g$, $(v_i)_{i=1}^g$,  $(u'_i)_{i=1}^g$, $(v'_i)_{i=1}^g$, be elements of $U(n)$ such that $\|\prod_{i=1}^{g} [u_i,v_i]-1_n\|<1/5g$, $\|u_i-u_i'\|<1/5g$ and  $\|v_i-v_i'\|<1/5g$ for $i=1,...,g$. Then
  \[\kappa\left(\prod_{i=1}^{g} [u_i,v_i] \right)=\kappa\left(\prod_{i=1}^{g} [u'_i,v'_i] \right).\]
  It follows that if $\kappa\left(\prod_{i=1}^{g} [u_i,v_i] \right)\neq 0$, then $\prod_{i=1}^{g} [u'_i,v'_i]\neq 1_n$.
 \end{lemma}
 \begin{proof}
  Kazdan considers the continuous paths in $U(n)$
   \[u_i(t)=u_i\exp(t\log(u_i^{-1}u_i')),\,\,v_i(t)=v_i\exp(t\log(v_i^{-1}v_i')),\, i=1,...,g.\]
   Then $\|u_i(t)-1_n\|<1/5g$, $\|v_i(t)-1_n\|<1/5g$, $t\in [0,1]$. It follows that $w(t)=\prod_{i=1}^{g} [u_i(t),v_i(t)]$ is a continuous path in $SU(n)$ such that
   $w(0)=\prod_{i=1}^{g} [u_i,v_i]$, $w(1)=\prod_{i=1}^{g} [u'_i,v'_i]$ and $\|w(t)-1_n\|<1$ for all $t\in [0,1]$. One concludes that $\kappa(w(0))=\kappa(w(1))$  since $t\mapsto\kappa(w(t))\in \Z$ is continuous.
 \end{proof}
 \begin{example}
  Kazdan's and Voiculescu's examples involve the sequence of pairs of unitaries
$$
	u_n=\left(\begin{matrix}
	0 & 0 &  0  &\cdots &{1} \cr
	{1}& 0 &  0  &\cdots & 0 \cr
	0 &{1} &  0  &\cdots & 0 \cr
	&  &  &\cdots & \cr
	0 &  0&  \cdots  &{1}& 0 \cr
	\end{matrix}\right),\,\,
	v_n=\left(\begin{matrix}
	{\lambda_n} & 0 &  0  & 0 & 0\cr
	0 & {\lambda_n^2} &  0  & 0 & 0\cr
	0 &  0&  \lambda_n^3 &{\cdot}& 0\cr
 &  &  &\cdots & \cr
	0 &  0&  0 & {\cdots}& {\lambda_n^n}\cr
	\end{matrix}\right),\quad \lambda_n=e^{2\pi i/n}
 $$
 \[[u_n,v_n]=\exp(-2\pi i/n)\cdot 1_n, \quad \|[u_n,v_n]-1_n\|=|\exp(2\pi i/n)-1|<2\pi/n\]
 \[\kappa([u_n,v_n])=\kappa(\exp(-2\pi i/n)1_n)=\frac{1}{2\pi i}\mathrm{Tr}(\log (\exp(-2\pi i/n)1_n))=-1.\]
 As noted in \cite{Kazhdan-epsilon} and rediscovered in \cite{Exel-Loring:winding}, Lemma~\ref{lemma:Kaz} implies that the sequence of pairs of  unitaries $u_n$ and $v_n$ does not admit commuting approximants.
 \end{example}
\begin{remark}\label{rem:push}
 Suppose that $\{\pi_n:A \to D_n\}_n$ is a bounded asymptotic homomorphism
of unital C*-algebras. Thus $\lim_{n \to \infty} \|\pi_n(aa')-\pi_n(a)\pi_n(a')\|=0$ for all $a,a'\in A$. The sequence $\{\pi_n\}_{n}$ induces a unital $*$-homomorphism
$A \to \prod_n D_n/\bigoplus_n D_n$ and hence a group homomorphism
$K_0(A)\to \prod_n K_0(D_n)/\bigoplus_n K_0(D_n)$.
This gives a canonical way to push forward  an element $x\in K_0(A)$ to a sequence $(\pi_{\nsh }(x))_n$ with components in $K_0(D_n)$ which is well-defined up to tail equivalence: two sequences are tail equivalent,
$(y_n)\equiv (z_n)$, if there is $m$ such that $x_n=y_n$ for all $n\geq m$. Note that $\pi_{\nsh }(x+x')\equiv \pi_{\nsh }(x)+\pi_{\nsh }(x')$.
Of course, if $\pi_n$ are genuine $*$-homomorphisms then $\pi_{\nsh }(x)=\pi_{n\, * }(x)$.
One can extend these considerations to Banach algebras. Occasionally it is convenient to work with a local version of this construction. For instance, if $\pi: A \to B$ is a unital linear contraction which is almost multiplicative in the sense that $\|\pi(aa')-\pi(a)\pi(a')\|<\ep$ for $a,a'$ in a finite subset $S$ of $A$, then one can pushforward specific projections $p$ in matrices over $A$ to  projections in matrices over $B$. Assuming that $S$ is sufficiently large and $\ep$ is sufficiently small, $(\pi\otimes id)(p)$ is close to a projection (use analytic functional calculus) whose K-theory class is denoted by $\pi_\sharp(p)$. Moreover given $p$ and $q$ with $[p]=[q]\in K_0(A)$, it will follow that $\pi_\sharp(p)=\pi_\sharp(q)$ provided that $S$ is sufficiently large and $\ep$ is sufficiently small.
Using this observation, we will sometimes abuse notation and write $\pi_\sharp(x)$ for $\pi_\sharp(p)-\pi_\sharp(p')$ where $x=[p]-[p']\in K_0(A)$ and the representatives $p,p'$ are fixed.
\end{remark}

Let $\G$ be a discrete countable group with classifying space  $B\G$. If $B\G$ is written as an increasing union of  finite simplicial complexes $Y_i$, then the K-homology of $B\G$ is $RK_0(B\G)\cong \varinjlim_i K_0(Y_i)$. 
Let $\mu^\G: RK_0(B\G) \to K_0(C^*(\G))$ denote the full assembly map \cite{Kas:inv}. Let $j:\ell^1(\G) \to C^*(\G)$ be the canonical homomorphism. There is a factorization of $\mu^\G$ through its $\ell^1$-version \cite{Lafforgue}:
\begin{equation}\xymatrix{
RK_0(B\G)\ar[dr]_{\mu^\G}\ar[r]^{\mu_1^\G}& K_0(\ell^1(\G))\ar[d]^{j_*}
\\
{}  &K_0(C^*(\G))
}
\end{equation}

A unital map $\pi:\G \to U(n)$ is call a quasi-representation of $\G$.
It induces a linear contraction $\pi:\ell^1(\G) \to M_n.$
Let $S\subset \G$ be a symmetric finite subset and let $\ep>0$.
We say that $\pi$ is $(S,\ep)$-multiplicative if $\|\pi(st)-\pi(s)\pi(t)\|<\ep$ for all $s,t \in S$.
Since $S$ is symmetric we see that
$\|\pi(s^{-1})-\pi(s)^*\|< \ep$ for all $s\in S$.
One can use sufficiently multiplicative  quasi-representations $\pi$  to pushforward K-theory elements of $K_0(\ell^1(\G))$ via a partially defined map $\pi_\sharp : K_0(\ell^1(\G))\to \Z$ as discussed in Remark~\ref{rem:push}. By Lemma 3.3. of \cite{BB}, if $x,y\in K_0(\ell^1(\G))$ are such that $j_*(x)=j_*(y)\in K_0(C^*(\G))$, then
$\pi_\sharp(x)=\pi_\sharp(y)$ provided that $\pi$ is sufficiently multiplicative.

A one-relator group is a group with a presentation of the form $\langle S; r \rangle$, where $r$ is single element in the free group $F(S)$ on the countable generating set $S$. An important example is the surface group
 \[\Gamma_g=\pi_1(\Sigma_g)=\langle s_1,t_1,...,s_g,t_g\,;\, \prod_{i=1}^g [s_i,t_i]\,\rangle,\]
where  $\Sigma_g=B\G_g$ is  a connected closed orientable surface of genus $g\geq 1$. We regard $s_i,t_i$ as the generators of the free group $\mathbb{F}_{2g}$. Their images in $\Gamma_g$ are denoted by $\bar{s}_i,\bar{t_i}$,  so that $\prod_{i=1}^g [\bar{s}_i,\bar{t_i}]=1.$

Let $[\Sigma_g]_K$ denote the  fundamental class of $\Sigma_g$ in K-homology.  It is independent of the choice of the spin structure of $\Sigma_g$ and $K_0(\Sigma_g)\cong \Z \oplus \widetilde{K}_0(\Sigma_g)\cong \Z \oplus \Z[\Sigma_g]_K$.
 In \cite{BB}, we extended
the Exel-Loring formula from $\Z^2$ to all surface groups $\G_g$, $g\geq 1$ as follows:
\begin{theorem}[Thm.4.2, \cite{BB}]\label{thm:bb}
There exist a finite set $S \subset \G_g$  and $\varepsilon>0$ such that if $\rho: \Gamma_g \to U(n)$ is any $(S,\ep)$-multiplicative quasi-representation, then
\begin{equation}\label{eq:ind-surf}
\rho_\sharp( \mu^{\G_g}[\Sigma_g])=
-\frac{1}{2\pi i}\mathrm{Tr}\left( \log \left(\prod_{i=1}^{g} [\pi(\bar{s}_i),\pi(\bar{t}_i)]\right)\right)
\end{equation}
\end{theorem}
The result above was extended to quasi-representations $\rho: \Gamma_g \to U(A)$ for $A$ a unital tracial $C^*$-algebra in \cite{Carrion-Dadarlat}.
\begin{theorem}[Thm.2.3, \cite{Carrion-Dadarlat}]\label{thm:cc}
There exist a finite set $S \subset \G_g$  and $\varepsilon>0$ such that if $\rho: \Gamma_g \to U(A)$ is any $(S,\ep)$-multiplicative quasi-representation, then
\begin{equation}\label{eq:ind-surfc}
\tau_*(\rho_\sharp( \mu^{\G_g}[\Sigma_g]))=
-\frac{1}{2\pi i}\tau\left( \log \left(\prod_{i=1}^{g} [\pi(\bar{s}_i),\pi(\bar{t}_i)]\right)\right)
\end{equation}
\end{theorem}
Here $\tau_*:K_0(A)\to \R$ is the homomorphism induced by $\tau$.
\begin{remark}
  The formula \eqref{eq:ind-surfc} was stated in \cite{Carrion-Dadarlat} without the negative sign. This was due to an inadvertent omission of the sign in the statement of Theorem 5.2 from \cite{Carrion-Dadarlat}, even though the correct sign was obtained in its proof.
\end{remark}

We are going to show that the formulae \eqref{eq:ind-surf}, \eqref{eq:ind-surfc} can be generalized to  arbitrary countable discrete groups, as stated in  Theorem~\ref{thm:index00}.

For a connected pointed CW complex $X$ there is a natural homomorphism $\beta^X:H_2(X,\Z)\to RK_0(X)$ which is a rational right inverse of  the Chern character in the sense that: $(ch_2\otimes \mathrm{id}_\Q)\circ (\beta^X\otimes \mathrm{id}_\Q)=\mathrm{id}_{H_2(X,\Q)}$ and hence it is rationally injective, see \cite{Bettaieb-Matthey-Valette} and
\cite{MR1951251}. The map $\beta^X$ is defined by composing the isomorphisms $H_2(X,\Z)\cong H_2(X^{(3)},\Z) \cong RK_0(X^{(3)})$
with the map $RK_0(X^{(3)})\to RK_0(X)$ induced by the inclusion of the 3-skeleton $X^{(3)}\hookrightarrow X$.

Let $\G$ be a countable discrete group.
For $X=B\G$, we denote by $\beta^{\G}$ the corresponding (rationally injective) homomorphism, \cite{MR1951251}, $$\beta^{\G}:  H_2(\G,\Z) \cong H_2(B\G,\Z) \to RK_0(B\G).$$
 Consider the map $\alpha^\G : H_2(\G,\Z) \to K_0(\ell^1(\G))$ defined by $\alpha^\G=\mu_1^\G  \circ \beta^\G$:
   \[\xymatrix{
\alpha^\G :H_2(\G,\Z)\ar[r]^{\beta^\G}& RK_0(B\G) \ar[r]^{\mu_1^\G}&K_0(\ell^1(\G))
}
\]
Chose a free resolution of $\G$:
\begin{equation}\label{eq:resol}
0 \to R \to F \stackrel{q}{\longrightarrow} \G \to 0, \quad q(a)=\bar{a},
\end{equation}
where $F$ and $R$ are free groups.
By Hopf's formula \cite{Brown:book-cohomology}, $$H_2(\G,\Z)=\frac{R\cap [F,F]}{[R,F]}.$$
Thus each element $x \in H_2(\G,\Z)$ is represented by a product of commutators, $\prod_{i=1}^{g} [a_i,b_i]$
 with $a_i,b_i \in F$ for some integer $g\geq 1$ and such that  $\prod_{i=1}^{g} [\bar{a}_i,\bar{b}_i]=1$.

\vskip 8pt
\centerline{\emph{Proof of Theorem~\ref{thm:index00}}}
\begin{proof}
We shall prove  only \eqref{eq:index}. The proof of \eqref{eq:indexc} is entirely similar except that one uses Theorem~\ref{thm:cc} instead of  Theorem~\ref{thm:bb} and $\kappa_\tau$ instead of $\kappa$.

Let $x\in H_2(\G,\Z)$ be represented by  a product of commutators
 $\prod_{i=1}^{g} [a_i,b_i]$ with $a_i,b_i \in F$ with $F$ as in \eqref{eq:resol}.
Let us recall that in the case of surface groups $\G_g$, with  resolution
\begin{equation}\label{eq:resolg}
0 \to R_{2g} \to F_{2g} \stackrel{q}{\longrightarrow} \G_g \to 0,
\end{equation}
it was shown in \cite[2.2.4]{Loday} that under the isomorphism
 $$H_2(\Sigma_g,\Z)\cong H_2(\G_g,\Z)=\frac{R_{2g}\cap [F_{2g},F_{2g}]}{[R_{2g},F_{2g}]},$$
 the fundamental class $[\Sigma_g]$ of $H_2(\Sigma_g,\Z)$ corresponds to the element $-x_g\in H_2(\G_g,\Z)$ where $x_g$ is the class of $\prod_{i=1}^g [s_i,t_i].$
 Following {Loday}, we
consider the homomorphism $F_{2g} \to F$ which maps $s_i$ to $a_i$ and $t_i$ to $b_i$. This induces an homomorphism $f:\G_g \to \G$ such that $f(\bar{s}_i)=\bar{a}_i$ and $f(\bar{t}_i)=\bar{b}_i$, $i=1,...,g$ and the corresponding map
$Bf:B\G_g \to B\G$. We make the identification $\Sigma_g=B\G_g$.
 If $[\Sigma_g]$ denotes the fundamental class of $H_2(\Sigma_g;\Z)$
then $\beta^{\Sigma_g}([\Sigma_g])=[\Sigma_g]_K$, see \cite[p.324]{MR2041902}.
From the previous discussion  we then obtain $\beta^{\G_g}(x_g)=-[\Sigma_g]_K$ and hence we can rewrite equation \eqref{eq:ind-surf} as
\begin{equation}\label{eq:ind-surff}
\rho_\sharp( \alpha^{\G_g}(x_g))=\kappa\left(\prod_{i=1}^{g} [\rho(\bar{s}_i),\rho(\bar{t}_i)] \right).
\end{equation}

 By naturality of $\beta$, \cite{MR2041902}  and $\mu$, \cite{BCH}, \cite{Lafforgue}, the following diagram is commutative.
  \[\xymatrix{
H_2(\G_g,\Z)\ar[r]^{\alpha^{\G_g}} \ar[d]_{f_*}& K_0(\ell^1(\G_g))\ar[d]_{f_*}\ar[dr]^{(\pi \circ f)_\sharp}&
\\
H_2(\G,\Z)\ar[r]^{\alpha^\G}& K_0(\ell^1(\G))\ar[r]_-{\pi_\sharp}& \Z
}
\]
Since $x_g$ is the generator of $H_2(\G_g,\Z)$ given by the product $\prod_{i=1}^g [s_i,t_i]$,  it follows that $f_*(x_g)=x$.
By fixing representatives of the relevant $K$-theory classes and by choosing $S$ sufficiently large and $\ep$ sufficiently small we may arrange that $\pi_\sharp(f_*(y))=(\pi \circ f)_\sharp(y)$
for finitely many elements $y\in K_0(\ell^1(\G_g))$ and in particular for $y=\alpha^{\G_g}(x_g)$.
 Thus:
\begin{equation}\label{eqn:long}
\pi_\sharp(\alpha^\G(x))=\pi_\sharp(\alpha^\G(f_*(x_g)))=\pi_\sharp(f_*(\alpha^{\G_g}(x_g)))=(\pi \circ f)_\sharp(\alpha^{\G_g}(x_g)).\end{equation}
On the other hand, the formula~\eqref{eq:ind-surff} applied for the quasi-representation $\rho=\pi\circ f:\Gamma_g \to U(n)$ implies  that
  \begin{equation}\label{eqn:longg}
(\pi \circ f)_\sharp(\alpha^{\G_g}(x_g))=
\kappa \left(\prod_{i=1}^{g} [\pi(f(\bar{s}_i)),\pi(f(\bar{t}_i))] \right)
\end{equation}
Since $f(\bar{s}_i)=\bar{a}_i$ and $f(\bar{t}_i)=\bar{b}_i$ we obtain from \eqref{eqn:long} and  \eqref{eqn:longg} that
  \begin{equation*}
\pi_\sharp(\alpha^\G(x))=
  \kappa\left(\prod_{i=1}^{g}[\pi(\bar{a}_i),\pi(\bar{b}_i)] \right).\end{equation*}
\end{proof}
\begin{remark}\label{rem:2.5}
The integer $\mathrm{wn}\det \left((1-t)1_n+t\prod_{i=1}^{g} [\pi(\bar{a}_i),\pi(\bar{b}_i)]\right)$ depends only on  the class $x$
  of $\prod_{i=1}^{g} [a_i,b_i]$ in $H_2(\G,\Z)$. This means that if we represent $x$ by a different product of commutators, $\prod_{i=1}^{g'} [a'_i,b'_i]$, then
  \[\kappa\left(\prod_{i=1}^{g}[\pi(\bar{a}_i),\pi(\bar{b}_i)] \right)=\kappa\left(\prod_{i=1}^{g'}[\pi(\bar{a}'_i),\pi(\bar{b}'_i)] \right)\]
  for all sufficiently multiplicative quasi-representations $\pi,$ since both this integers are equal to $\pi_\sharp(\alpha^\G(x))$ by equation~\ref{eq:index}.
\end{remark}
\section{Quasi-representations with nontrivial invariants}
Our next goal is to exhibit classes of groups that admit quasi-representations $\pi$ for which the invariants from Theorem~\ref{thm:index00} do not vanish.
This is addressed in Theorem~\ref{thm:2}.

Let $\G$ be a discrete countable group. Let $\QQ$ be the universal UHF-algebra, $\QQ\cong \bigotimes_{n\geq 1} M_n(\C)$. Consider the natural  pairing $$KK(\C,C^*(\G))\times KK(C^*(\G),\QQ) \to KK(\C,\QQ)\cong \Q,$$ given by
$(x,y)\mapsto x\otimes_{C^*(\G)} y$. Consider  the full assembly map $\mu:RK_0(B\G)\to K_0(C^*(\G))$ and the dual assembly map with rational coefficients
$\nu: KK(C^*(\G),\QQ) \to RK^0(B\G,\Q)$, \cite{Kas:inv}, \cite{Kasparov:conspectus}. For each finite CW complex $Y \subset B\G$, let $\nu_Y: KK(C^*(\G),\QQ) \to RK^0(B\G,\Q) \to K^0(Y,\Q)$ be the composition of $\nu$ with the restriction map  $RK^0(B\G,\Q)\to K^0(Y,\Q)$.
Let $\mu_Y:K_0(Y)\to K_0(C^*(\G))$ be the composition of $\mu$ with $K_0(Y)\to K_0(B\G)$.
By \cite[6.2]{Kas:inv} these maps satisfy the identity:
\begin{equation}\label{eq;duality}
  \nu_Y(y)\otimes_{C(Y)} z = \mu_Y(z)\otimes_{C^*(\G)} y
\end{equation}
for all $z\in K_0(Y)$ and $y\in KK(C^*(\G),\QQ)$.

If $B\G$ is written as the union of an increasing sequence
$(Y_i)_{i}$ of finite CW complexes, then as explained in the proof of Lemma 3.4 from \cite{Kasparov-Skandalis-kk}, there is a Milnor $\varprojlim^1$ exact sequence which implies that
\begin{equation}\label{eq:milnor}
RK^0(B\G;\Q)\cong \varprojlim K^0(Y_i;\Q).
\end{equation}
We denote $\nu_{Y_i}$ by $\nu_i$ and $\mu_{Y_i}$ by $\mu_i$.
On the other hand $RK_0(B\G)=\varinjlim K_0(Y_i)$ and $\mu$ is just the limit of the compatible maps $\mu_i:K_0(Y_i)\to K_0(C^*(\G))$.
Using \eqref{eq;duality}
  we deduce that the following diagram is commutative
\begin{equation}\label{diag:i}\xymatrix{
KK(C^*(\G),\QQ)\ar@{->}[d]_{\nu_i}\ar[r]& \Hom(K_0(C^*(\G)),\Q)\ar@{->}[d]^{\mu_i^*}
\\
{RK^0(Y_i;\Q)}\ar@{->}[r]_-{\delta_i}  &\Hom(RK_0(Y_i),\Q)
}
\end{equation}
where
the horizontal arrows correspond to natural pairings of K-theory with
K-homology.

Let $\EG$ be the classifying space for proper actions of $\G$, \cite{BCH}. It is known that $\EG$ admits a locally compact model, \cite{Kasparov-Skandalis-Annals}.
Let us recall  that $\G$ has a $\gamma$-element if there exists a $\G-C_0(\EG)$-algebra $A$ in the sense of Kasparov \cite{Kas:inv} and two elements $d\in KK_\G(A,\C)$ and $\eta\in KK_G(\C,A)$ (called Dirac and dual-Dirac elements, respectively) such that the element $\gamma=\eta\otimes_Ad\in KK_\G(\C,\C)$ has the property that $p^*(\gamma)=1\in \mathcal{R}KK_G(\EG;C_0(\EG),C_0(\EG))$ where
$p:\EG \to \text{point}$, \cite{Tu:gamma}. We refer the reader to \cite{Kas:inv} for the definitions and the basic properties of these groups. The groups which are coarsely embeddable in a Hilbert space admit a $\gamma$-element, \cite{Tu:gamma}. The class of groups which are coarsely embeddable in a Hilbert space
include the amenable groups, the exact (boundary amenable) groups, the linear groups and the hyperbolic groups.
\begin{proposition}\label{prop:diag}
  Suppose that $\G$ has a $\gamma$-element. Then for any homomorphism $h: K_0(C^*(\G)) \to \Q$ there is $y\in KK(C^*(\G),\QQ)$ such that
  $h(\mu(z))=\mu(z)\otimes_{C^*(\G)} y$ for all $z\in RK_0(B\G)$.
\end{proposition}
\begin{proof}
Since $RK^0(B\G;\Q)\cong \varprojlim K^0(Y_i;\Q)$ and $RK_0(B\G)=\varinjlim K_0(Y_i)$, after passing to limit in \eqref{diag:i},
  we deduce that the following diagram is commutative
\begin{equation}\xymatrix{
KK(C^*(\G),\QQ)\ar@{->>}[d]_{\nu}\ar[r]& \Hom(K_0(C^*(\G)),\Q)\ar@{->>}[d]^{\mu^*}
\\
{RK^0(B\G;\Q)}\ar@{->>}[r]_-\delta  &\Hom(RK_0(B\G),\Q)
}
\end{equation}
The horizontal arrows correspond to natural pairings of K-theory with
K-homology.

The map $\delta$  is surjective by Lemma 3.4 of \cite{Kasparov-Skandalis-kk}. If $\G$ has a $\gamma$-element it is known that the vertical maps are surjective as well. Indeed $\mu$ is rationally injective by \cite{Ska-Tu-Yu:BC}, \cite{Tu:gamma} and hence $\mu^*$ is surjective. For the surjectivity of $\nu$ (due to Kasparov) see \cite[Cor.4.2]{CCC}.
Let $h\in \Hom(K_0(C^*(\G)),\Q)$. Then $h \circ \mu \in \Hom(RK_0(B\G),\Q)$. Since both $\nu$ and $\delta$ are surjective, there is $y\in KK(C^*(\G),\QQ)$ such that $\delta(\nu(y))=h \circ \mu$.

Thus $\delta(\nu(y))=h \circ \mu$ implies that $\delta_i(\nu_i(y))=h \circ \mu_i$ for some $i_0$ and hence for all indices  $i\geq i_0$.
Every $z\in RK_0(B\G)$ is the image of some $z_i\in K_0(Y_i)$ with $i \geq i_0$. It follows from \eqref{eq;duality} that
$$h \circ \mu_i(z_i)=\delta_i(\nu_i(y))(z_i)=\nu_i(y)\otimes_{C(Y_i)} z_i = \mu_i(z_i)\otimes_{C^*(\G)} y, $$ and hence $h(\mu(z))=\mu(z)\otimes_{C^*(\G)} y$. \qedhere
\end{proof}

A countable discrete group $G$ is \emph{quasidiagonal} if it is isomorphic to a subgroup of the unitary group of a quasidiagonal $C^*$-algebra \cite{CCC}. Equivalently, there is a faithful representation $\pi:\G \to U(H)$ on a Hilbert space for which  there is an increasing sequence  $(p_n)_{n}$ of finite dimensional projections which converges strongly to $1_H$ and such that $\lim_{n\to \infty}\|[\pi(s),{p_n}]\|= 0$ for all $s \in \G$.
Thus, a maximally almost periodic group (MAP) is quasidiagonal.
  Amenable groups, or more generally, residually amenable groups are also quasidiagonal as a consequence of \cite{TWW}.

If $\G$ has a $\gamma$-element, then it is known that $\mu^\G$ is rationally injective \cite{Tu:BC} and therefore so is the map $\bar{\alpha}^\G: H_2(\G,\Z) \to K_0(C^*(\G))$ defined by $\bar{\alpha}^\G=\mu^\G \circ \beta^{\G}=j_*\circ \alpha^{\G}$, where $j_*:K_0(\ell^1(\G))\to K_0(C^*(\G))$. We shall use notation as in \eqref{eq:resol}.

\begin{theorem}\label{thm:2}
 Let $\G$ be a quasidiagonal group which admits a $\gamma$-element. Suppose that $x$ is a non-torsion element of $H_2(\G,\Z)$ represented by a product of commutators
 $\prod_{i=1}^{g} [a_i,b_i]$ with $a_i,b_i \in F$ and $\prod_{i=1}^{g} [\bar{a}_i,\bar{b}_i]=1$. Then there is an asymptotic homomorphism $\{\pi_n: \G \to U(k_n)\}_n$
  such that $$\mathrm{wn}\det \left((1-t)1_{k_n}+t\prod_{i=1}^{g} [\pi_n(\bar{a}_i),\pi_n(\bar{b}_i)]\right)\neq 0$$ for all sufficiently large $n$.
\end{theorem}
\begin{proof}
 Let us recall that $\alpha^\G=\mu_1^\G\circ\beta^\G$ and $\bar{\alpha}^\G=\mu^\G\circ\beta^\G$. By Theorem~\ref{thm:index00} it suffices to find $(\pi_n)_n$ such that $(\pi_n)_\sharp (\alpha^\G(x))\neq 0$  for all sufficiently large $n$. We claim that  it suffices to find a unital completely positive (ucp) asymptotic morphism $\{\psi_n: C^*(\G )\to M_{k_n}\}_n$ such that $(\psi_n)_\sharp (\bar{\alpha}^\G(x))\neq 0$  for all sufficiently large $n$. Indeed, by functional calculus one can perturb the restriction to $\G$ of each $\psi_n$ to a unital map $\pi_n:\G\to U(k_n)$ such that $\lim_n \|\pi_n(s)-\psi_n(s)\|=0$ for all $s\in \G$. Then the  asymptotic homomorphism $\{\pi_n: \G \to U(k_n)\}_n$ induces  $*$-homomorphisms $\bm{\pi}: \ell^1(\G) \to \prod_n M_{k_n}/\bigoplus_n M_{k_n}$ and $\bm{\underline{\pi}}:C^*(\G)\to \prod_n M_{k_n}/\bigoplus_n M_{k_n}$ with
 $j\circ \bm{\underline{\pi}}=\bm{\pi}$ such that $\bm{\underline{\pi}}$ is equal to the $*$-homomorphism induced by $\{\psi_n\}_n$.
 It follows that  $(\pi_n)_\sharp ({\alpha}^\G(x))=(\psi_n)_\sharp (\bar{\alpha}^\G(x))\neq 0$ for all sufficiently large $n$.

  Since $x$ is a non-torsion element  and since $\bar{\alpha}^\G$ is a composition of rationally injective maps ($\G$ has a $\gamma$-element),
 there is $h: K_0(C^*(\G)) \to \Q$ such that $h(\bar{\alpha}^\G(x))\neq 0$.
  Since $\G$ has a $\gamma$-element and it is quasidiagonal, it follows by \cite[Thm.4.6]{CCC} that
 \(\nu(KK(C^*(\G),\mathcal{Q})_{qd})= \nu(KK(C^*(\G),\mathcal{Q}))=RK^0(B\G;\Q)  \). Therefore in the proof of Proposition~\ref{prop:diag} we can choose $y\in KK(C^*(\G),\mathcal{Q})_{qd}$ such that
$h(\mu(z))=\mu(z)\otimes_{C^*(\G)} y$ for all $z\in RK_0(B\G)$. In particular, we obtain that
$h(\bar{\alpha}^\G(x))=\bar{\alpha}^\G(x)\otimes_{C^*(\G)} y \neq 0.$
Since $y\in KK(C^*(G),\QQ)_{qd}$, $y$ is represented  by  a pair of nonzero $*$-representations $\varphi,\psi:C^*(\G)\to M(K(H)\otimes \QQ)$,  such that $\varphi(a)-\psi(a)\in K(H)\otimes \QQ$,
 $a\in C^*(\G)$, and with property that
 there is an increasing approximate unit $(p_n)_n$ of $K(H)$ consisting of projections such that $(p_n\otimes 1_\QQ)_n$ commutes asymptotically with both $\varphi(a)$ and $\psi(a)$, for all $a\in C^*(\G)$, see \cite[Def.4.4]{CCC}.
  It is then clear that $\varphi^{(0)}_n=(p_n\otimes 1_\QQ) \varphi(p_n\otimes 1_\QQ)$ and $\varphi^{(1)}_n=(p_n\otimes 1_\QQ) \psi(p_n\otimes 1_\QQ)$
 are contractive completely positive asymptotic homomorphisms from  $C^*(\G)$ to $K(H)\otimes \QQ$. Let $1$ denote the unit of $C^*(\G)$.
It is routine to  further perturb these maps to completely positive asymptotic homomorphisms such that
$\varphi^{(r)}_n(1)$, $r=0,1$, are projections so that we can view this maps as ucp maps into matrix subalgebras of $\QQ$.
By \cite[Prop.2.5]{AA} the Kasparov product $\bar{\alpha}^\G(x)\otimes_{C^*(\G)} y$ can be computed as
\begin{equation}
(\varphi^{(0)}_n)_\sharp(\bar{\alpha}^\G(x))-(\varphi^{(1)}_n)_\sharp(\bar{\alpha}^\G(x))\equiv \bar{\alpha}^\G(x)\otimes_{C^*(\G)} y \neq 0.
\end{equation}
It follows that there is $n_0$ such that for each $n\geq n_0$ there is $r_n\in \{0,1\}$ such that $(\varphi^{(r_n)}_n)_\sharp(\bar{\alpha}^\G(x))$ is nonzero.
Then $\psi_n:=\varphi^{(r_n)}_n$ has the desired properties.
\end{proof}
Any finitely generated linear group $\Gamma$ is residually finite by Malcev's theorem  and exact by \cite{GuenHW} and so it satisfies the hypotheses of Theorem~\ref{thm:2}.
In particular, this is the case  for finitely generated torsion free nilpotent groups \cite{CCC}.

 \begin{corollary}\label{cor:2}
Let $\G$ be a quasidiagonal group which admits a $\gamma$-element and such that $H_2(\G,\Q)\neq 0$.  Then there is an asymptotic homomorphism $\{\pi_n: \G \to U(k_n)\}_n$ for which there exist no genuine representations $\{\pi'_n: \G \to U(k_n)\}_n$
  such that $\lim_{n\to \infty} \|\pi_n(s)-\pi'_n(s)\|=0$ for all $s\in \G.$
\end{corollary}
\begin{proof}
  This follows from Theorem~\ref{thm:2} and Lemma \ref{lemma:Kaz} as
   \[\kappa\left(\prod_{i=1}^{g} [\pi'_n(\bar{a}_i),\pi'_n(\bar{b}_i)] \right)=0\]
    for genuine representations $\pi'_n$ of $\G$. A more  general result proved in \cite{CCC} asserts that it suffices to assume the nonvanishing of some $H_{2k}(\G,\Q)$, $k\geq 1$.
\end{proof}

\bibliographystyle{abbrv}
\smaller[1]

\begin{thebibliography}{ABC99}
\bibitem{BCH}
P.~Baum, A.~Connes, and N.~Higson.
\newblock Classifying space for proper actions and {$K$}-theory of group
  {$C^\ast$}-algebras.
\newblock In {\em {$C^\ast$}-algebras: 1943--1993 ({S}an {A}ntonio, {TX},
  1993)}, volume 167 of {\em Contemp. Math.}, pages 240--291. Amer. Math. Soc.,
  Providence, RI, 1994.

\bibitem{Bettaieb-Matthey-Valette}
H.~Bettaieb, M.~Matthey, and A.~Valette.
\newblock Unbounded symmetric operators in {$K$}-homology and the
  {B}aum-{C}onnes conjecture.
\newblock {\em J. Funct. Anal.}, 229(1):184--237, 2005.

\bibitem{Brown:book-cohomology}
K.~S. Brown.
\newblock {\em Cohomology of groups}, volume~87 of {\em Graduate Texts in
  Mathematics}.
\newblock Springer-Verlag, New York, 1994.
\newblock Corrected reprint of the 1982 original.

\bibitem{Carrion-Dadarlat}
J.~R. Carri{\'o}n and M.~Dadarlat.
\newblock Quasi-representations of surface groups.
\newblock {\em J. Lond. Math. Soc. (2)}, 88(2):501--522, 2013.

\bibitem{CGM:flat}
A.~Connes, M.~Gromov, and H.~Moscovici.
\newblock Conjecture de {N}ovikov et fibr\'es presque plats.
\newblock {\em C. R. Acad. Sci. Paris S\'er. I Math.}, 310(5):273--277, 1990.

\bibitem{BB}
M.~Dadarlat.
\newblock Group quasi-representations and index theory.
\newblock {\em J. Topol. Anal.}, 4(3):297--319, 2012.

\bibitem{AA}
M.~Dadarlat.
\newblock Group quasi-representations and almost flat bundles.
\newblock {\em J. Noncommut. Geom.}, 8(1):163--178, 2014.

\bibitem{CCC}
M.~Dadarlat.
\newblock Obstructions to matricial stability of discrete groups and almost
  flat {K}-theory.
\newblock {\em Adv. Math.}, 384:Paper No. 107722, 29, 2021.

\bibitem{ESS-published}
S.~Eilers, T.~Shulman, and A.~P.~W. S{\o}rensen.
\newblock {$C^*$}-stability of discrete groups.
\newblock {\em Adv. Math.}, 373:107324, 41, 2020.

\bibitem{Exel:Pacific}
R.~Exel.
\newblock The soft torus and applications to almost commuting matrices.
\newblock {\em Pacific J. Math.}, 160(2):207--217, 1993.

\bibitem{Exel-Loring:winding}
R.~Exel and T.~Loring.
\newblock Almost commuting unitary matrices.
\newblock {\em Proc. Amer. Math. Soc.}, 106(4):913--915, 1989.

\bibitem{Exel-Loring:inv=}
R.~Exel and T.~A. Loring.
\newblock Invariants of almost commuting unitaries.
\newblock {\em J. Funct. Anal.}, 95(2):364--376, 1991.

\bibitem{GuenHW}
E.~Guentner, N.~Higson, and S.~Weinberger.
\newblock The {N}ovikov conjecture for linear groups.
\newblock {\em Publ. Math. Inst. Hautes \'Etudes Sci.}, (101):243--268, 2005.

\bibitem{Kasparov-Skandalis-Annals}
G.~Kasparov and G.~Skandalis.
\newblock Groups acting properly on ``bolic'' spaces and the {N}ovikov
  conjecture.
\newblock {\em Ann. of Math. (2)}, 158(1):165--206, 2003.

\bibitem{Kas:inv}
G.~G. Kasparov.
\newblock Equivariant {$KK$}-theory and the {N}ovikov conjecture.
\newblock {\em Invent. Math.}, 91(1):147--201, 1988.

\bibitem{Kasparov:conspectus}
G.~G. Kasparov.
\newblock {$K$}-theory, group {$C^*$}-algebras, and higher signatures
  (conspectus).
\newblock In {\em Novikov conjectures, index theorems and rigidity, {V}ol.\ 1
  ({O}berwolfach, 1993)}, volume 226 of {\em London Math. Soc. Lecture Note
  Ser.}, pages 101--146. Cambridge Univ. Press, Cambridge, 1995.

\bibitem{Kasparov-Skandalis-kk}
G.~G. Kasparov and G.~Skandalis.
\newblock Groups acting on buildings, operator {$K$}-theory, and {N}ovikov's
  conjecture.
\newblock {\em $K$-Theory}, 4(4):303--337, 1991.

\bibitem{Kazhdan-epsilon}
D.~Kazhdan.
\newblock On {$\varepsilon $}-representations.
\newblock {\em Israel J. Math.}, 43(4):315--323, 1982.

\bibitem{Lafforgue}
V.~Lafforgue.
\newblock {$K$}-th\'eorie bivariante pour les alg\`ebres de {B}anach et
  conjecture de {B}aum-{C}onnes.
\newblock {\em Invent. Math.}, 149(1):1--95, 2002.

\bibitem{Loday}
J.-L. Loday.
\newblock {$K$}-th\'{e}orie alg\'{e}brique et repr\'{e}sentations de groupes.
\newblock {\em Ann. Sci. \'{E}cole Norm. Sup. (4)}, 9(3):309--377, 1976.

\bibitem{Loring:k(uv)-paper}
T.~A. Loring.
\newblock {$K$}-theory and asymptotically commuting matrices.
\newblock {\em Canad. J. Math.}, 40(1):197--216, 1988.

\bibitem{MR1951251}
M.~Matthey.
\newblock Mapping the homology of a group to the {$K$}-theory of its
  {$C^*$}-algebra.
\newblock {\em Illinois J. Math.}, 46(3):953--977, 2002.

\bibitem{MR2041902}
M.~Matthey.
\newblock The {B}aum-{C}onnes assembly map, delocalization and the {C}hern
  character.
\newblock {\em Adv. Math.}, 183(2):316--379, 2004.

\bibitem{Ska-Tu-Yu:BC}
G.~Skandalis, J.~L. Tu, and G.~Yu.
\newblock The coarse {B}aum-{C}onnes conjecture and groupoids.
\newblock {\em Topology}, 41(4):807--834, 2002.

\bibitem{TWW}
A.~Tikuisis, S.~White, and W.~Winter.
\newblock Quasidiagonality of nuclear {$C^\ast$}-algebras.
\newblock {\em Ann. of Math. (2)}, 185(1):229--284, 2017.

\bibitem{Tu:BC}
J.-L. Tu.
\newblock La conjecture de {B}aum-{C}onnes pour les feuilletages moyennables.
\newblock {\em $K$-Theory}, 17(3):215--264, 1999.

\bibitem{Tu:gamma}
J.-L. Tu.
\newblock The gamma element for groups which admit a uniform embedding into
  {H}ilbert space.
\newblock In {\em Recent advances in operator theory, operator algebras, and
  their applications}, volume 153 of {\em Oper. Theory Adv. Appl.}, pages
  271--286. Birkh\"auser, Basel, 2005.

\bibitem{Voi:unitaries}
D.~Voiculescu.
\newblock {Asymptotically commuting finite rank unitary operators without
  commuting approximats}.
\newblock {\em Acta Sci. Math. (Szeged)}, 45:429--431, 1983.

\end{thebibliography}

\end{document}